\documentclass[12pt,oneside]{amsart}
\usepackage[utf8]{inputenc}
\usepackage{geometry}
\usepackage{verbatim}
\usepackage{microtype}
\usepackage{physics}
\usepackage{enumitem}

\usepackage[pdfusetitle]{hyperref}
\usepackage{color}
\hypersetup{colorlinks,linkcolor=blue,citecolor=cyan}
\usepackage{amsthm}
\usepackage[nameinlink]{cleveref}
\usepackage{amsmath,amssymb}

\usepackage{tikz}

\newtheorem{thm}{Theorem}[section]
\newtheorem{prop}[thm]{Proposition}
\newtheorem{cor}[thm]{Corollary}
\theoremstyle{definition}
\newtheorem{prob}[thm]{Problem}

\newcommand{\C}{\mb C}
\newcommand{\hra}{\hookrightarrow}
\newcommand{\mb}{\mathbb}
\newcommand{\mf}{\mathfrak}
\newcommand{\HH}{\mb H}
\newcommand{\N}{\mb N}
\newcommand{\p}{\mf p}
\newcommand{\q}{\mf q}
\newcommand{\Q}{\mb Q}
\newcommand{\R}{\mb R}
\newcommand{\thra}{\twoheadrightarrow}
\newcommand{\tl}{\mathrel \triangleleft}
\newcommand{\ve}{\varepsilon}
\newcommand{\Z}{\mb Z}

\title{A dichotomy for Polish modules}
\date{\today}

\author{Joshua Frisch}
\address{Department of Mathematics,
California Institute of Technology,
1200 E. California Blvd.,
Pasadena,
CA 91125}
\email{jfrisch@caltech.edu}

\author{Forte Shinko}
\address{Department of Mathematics,
California Institute of Technology,
1200 E. California Blvd.,
Pasadena,
CA 91125}
\email{fshinko@caltech.edu}

\thanks{The authors were partially supported by NSF Grant DMS-1950475.}

\begin{document}

\maketitle

\begin{abstract}
    Let $R$ be a ring equipped with a proper norm.
    We show that under suitable conditions on $R$,
    there is a natural basis under continuous linear injection
    for the set of Polish $R$-modules which are not countably generated.
    When $R$ is a division ring,
    this basis can be taken to be a singleton.
\end{abstract}

\section{Introduction}
The axiom of choice allows us to construct many abstract algebraic homomorphisms
between topological algebraic systems
which are incredibly non-constructive.
A longstanding theme in descriptive set theory
is to study to what extent we can,
and to what extent we provably cannot,
construct such homomorphisms in a ``definable'' way.
Here the notion of definability is context-dependent
but often includes continuous, Borel, or projective maps.

A classical example of such an abstract construction,
which provably cannot be constructed with ``nice'' sets
is the existence of a Hamel basis for $\R$ over $\Q$.
It is well-known that such a basis cannot be Borel,
or more generally, analytic.
Similar phenomena show up when constructing Hamel bases for topological vector spaces,
or constructing an isomorphism of the additive groups of $\R$ and $\C$.

A more recent theme in descriptive set theory
is that such undefinability criteria can often be leveraged in order to gain,
and hopefully utilize, additional structure.
For example,
Silver's theorem \cite{Sil80} and the Glimm-Effros dichotomy \cite{HKL90}
interpret the non-reducibility of Borel equivalence relations
not as a pathology but rather as the first step in the burgeoning theory
of invariant descriptive set theory
(see \cite{Gao09} for background).
Similarly,
work starting with \cite{KST99} studies and exploits the difference
between abstract chromatic numbers and more reasonably definable
(for example, continuous or Borel)
chromatic numbers. 
A key feature in many of these theories
(and all of the above examples)
is the existence of dichotomy theorems,
which state that either an object is simple,
or there is a canonical obstruction contained inside of it.
This is usually stated in terms of preorders,
saying that there is a natural basis for the preorder
of objects which are not simple
(recall that a \textbf{basis} for a preorder $P$
is a subset $B \subseteq P$ such that for every $p \in P$,
there is some $b\in B$ with $b \le p$).

In this paper,
we apply a descriptive set-theoretic approach to vector spaces and more generally,
modules, over a locally compact Polish ring\footnotemark.
\footnotetext{All rings will be assumed to be unital.}
For a Polish ring $R$,
a \textbf{Polish $R$-module}
is a topological left $R$-module
whose underlying topology is Polish.
Given Polish $R$-modules $M$ and $N$,
we say that $M$ \textbf{embeds} into $N$,
denoted $M \sqsubseteq^R N$,
if there is a continuous linear injection from $M$ into $N$.
One particularly nice aspect of Polish modules is that
the notion of ``definable'' reduction is much simpler than in the general case.
By Pettis's lemma,
any Baire-measurable homomorphism between Polish modules
is in fact automatically continuous
(see \cite[9.10]{Kec95}).
Thus there is no loss of generality in considering continuous homomorphisms
rather than a priori more general Borel homomorphisms.

Our main results give a dichotomy for Polish modules being countably generated.
More precisely,
we give a countable basis under $\sqsubseteq^R$
for Polish modules which are not countably generated.
While these results are stated in a substantial level of generality
(they are true for all left-Noetherian countable rings
and many Polish division rings),
we feel that the most interesting cases are over some of the most concrete rings.
For example,
over $\Q$,
we show the existence of a unique
(up to bi-embeddability)
minimal uncountable Polish vector space $\ell^1(\Q)$.
We further show that nothing bi-embeddable with
$\ell^1(\Q)$ is locally compact,
and thus that every uncountable-dimensional locally compact Polish vector space
(for example, $\R$)
is strictly more complicated than $\ell^1(\Q)$. 

Another case of particular interest is the case of $\Z$-modules,
that is, abelian groups. 
We show that there is a countable basis of minimal uncountable abelian Polish groups
(one for each prime number and one for characteristic $0$).
Furthermore,
there exists a maximal abelian Polish group by \cite{Shk99},
as well as many natural but incomparable elements
(for example,
$\Q_p$ and $\R$ are incomparable under $\sqsubseteq^\Q$
as are $\Q_p$ and $\Q_r$ for $p\neq r$). 

Our dichotomy theorems will hold for rings equipped with a proper norm.
A \textbf{(complete, proper) norm} on an abelian group $A$
is a function $\norm{\cdot} \colon A \to [0,\infty)$
such that the map $(a, b)\mapsto \norm{a - b}$
is a (complete, proper) metric on $A$
(recall that a metric is proper if every closed ball is compact).
A \textbf{norm} on a ring $R$
is a norm $|\cdot|$ on $(R, +)$
such that $|rs| \le |r| |s|$
for every $r, s \in R$.
A \textbf{proper normed ring}
is a ring equipped with a proper norm.
Every countable ring admits a proper norm
(see \Cref{ProperNorm}).
Given a proper normed ring $R$,
the $R$-module $\ell^1(R)$ is defined as follows:
\[
    \ell^1(R)
    = \left\{(r_k)_k\in R^\N
    : \sum_k \frac{|r_k|}{k!} < \infty\right\}
\]
(here,
$\frac{1}{k!}$ can be replaced with any summable sequence).
Then $\norm{(r_k)_k} := \sum_k \frac{|r_k|}{k!}$
is a complete separable norm on $(\ell^1(R), +)$,
turning $\ell^1(R)$ into a Polish $R$-module.

The following theorems will be obtained as
special cases of results in \Cref{proofs}.

A \textbf{division ring} is a ring $R$
such that every nonzero $r \in R$
has a two-sided inverse.
\begin{thm}\label{division}
    Let $R$ be a proper normed division ring
    and let $M$ be a Polish $R$-vector space.
    Then exactly one of the following holds:
    \begin{enumerate}[label=(\arabic*)]
        \item $\dim_R(M)$ is countable.
        \item $\ell^1(R) \sqsubseteq^R M$.
    \end{enumerate}
\end{thm}
This seems to be new,
even when $R$ is a finite field,
in which case $\ell^1(R) = R^\N$.
This also implies a special case of \cite[Theorem 24]{Mil12},
which says that if $\dim_R(M)$ is uncountable,
then there is a linearly independent perfect set
(see \Cref{miller}).

An analogous statement holds for
a large class of discrete rings.
A ring is \textbf{left-Noetherian}
if every increasing sequence of left ideals stabilizes.
\begin{thm}\label{discrete}
    Let $R$ be a left-Noetherian discrete proper normed ring
    and let $M$ be a Polish $R$-module.
    Then exactly one of the following holds:
    \begin{enumerate}[label=(\arabic*)]
        \item $M$ is countable.
        \item $\ell^1(S) \sqsubseteq^R M$ for some nonzero quotient $S$ of $R$.
    \end{enumerate}
\end{thm}
Note that this basis is countable since a countable left-Noetherian ring
only has countably many left ideals.

For abelian Polish groups,
we obtain an irreducible basis
(see \Cref{abelian}):
\begin{thm}\label{abelianMain}
    Let $A$ be an uncountable abelian Polish group.
    Then one of the following holds:
    \begin{enumerate}
        \item $\ell^1(\Z) \sqsubseteq^\Z A$.
        \item $(\Z/p\Z)^\N \sqsubseteq^\Z A$ for some prime $p$.
    \end{enumerate}
\end{thm}
Related statements have been shown by Solecki,
see \cite[Proposition 1.3, Theorem 1.7]{Sol99}.

The theorems in \Cref{proofs}
will be shown for a substantially broader class of modules.
In order to contextualize this,
we remark that considering even very basic module homomorphisms
(for example, the inclusion of $\Q$ into $\R$ as $\Q$-vector spaces)
naturally leads us to consider the broader class of quotients
of Polish modules by sufficiently definable submodules.
Such quotient modules are in general not Polish
(they are not necessarily even standard Borel)
but are still important objects of descriptive set-theoretic interest.
They play a crucial role in \cite{BLP20}
in the form of ``groups with a Polish cover'',
and they also form some of the most classical examples
of countable Borel equivalence relations
(for example,
the commensurability relation on the positive reals
naturally comes equipped with an abelian group structure).
The embedding order on quotient modules will be defined
analogously to the homomorphism reductions for Polish groups
studied in \cite{Ber14, Ber18}.

\subsection*{Acknowledgments}
We would like to thank Alexander Kechris,
S{\l}awomir Solecki,
and Todor Tsankov
for several helpful comments and remarks.
We would also like to thank the anonymous referee
for finding an error in an earlier draft,
as well as numerous helpful improvements.

\section{Polish modules}
Most Polish modules which cannot be written as direct sums,
even over a field.
This will follow from a more general statement about Polish groups.

Given a Polish group $G$,
a family $(H_x)_{x\in \R}$ of subgroups of $G$ is
\begin{enumerate}[label=(\roman*)]
    \item \textbf{analytic}
        if the set $\{(g, x) \in G\times \R : g \in H_x\}$ is analytic;
    \item \textbf{independent}
        if for every finite $F\subseteq \R$ and every $x \in \R\setminus F$,
        we have $H_x \cap \ev{H_y}_{y\in F} = 1$;
    \item \textbf{generating}
        if $(H_x)_{x\in \R}$ generates $G$.
\end{enumerate}
In particular,
if $G$ is the direct sum or the free product of $(H_x)_{x\in \R}$,
then $(H_x)_x$ is an independent generating family.

\begin{prop}
    Let $G$ be a Polish group,
    and let $(H_x)_{x\in \R}$ be an analytic independent
    generating family of subgroups of $G$.
    Then there are only countably many $x \in \R$ with $H_x$ nontrivial,
    and only finitely many $x \in \R$ with $H_x$ uncountable.
\end{prop}
\begin{proof}
    Let $A_n$ be the set of $g\in G$
    which can be written in the form $h_0 h_1 \cdots h_{n-1}$
    with each $h_i$ in some $H_x$.
    Then $A_n$ is analytic,
    and thus Baire-measurable.
    Since $G = \bigcup_n A_n$,
    there is some $A_n$ which is non-meager.
    By Pettis's lemma,
    we can replace $n$ with $2n$
    and assume that $A_n$ has non-empty interior.
    Thus $G$ can be covered by countably many right translates
    $(A_n g_k)_k$ of $A_n$.
    
    Let $X \subseteq \R$ be the set of $x \in \R$ with $H_x$ nontrivial,
    and suppose that $X$ is uncountable.
    For each $x \in X$,
    fix some nontrivial $h_x \in H_x$.
    Fix an equivalence relation $E$ on $X$
    with every class of cardinality $n + 1$.
    Then there must be two $E$-classes
    $(x_i)_{i \le n}$ and $(y_i)_{i \le n}$
    such that $h_{x_0} h_{x_1}\cdots h_{x_n}$
    and $h_{y_0} h_{y_1}\cdots h_{y_n}$
    are in the same $A_n g_k$.
    But then
    \[
        h_{x_0} h_{x_1}\cdots h_{x_n}
        (h_{y_0} h_{y_1}\cdots h_{y_n})^{-1}
        \in A_n g_k (A_n g_k)^{-1}
        = A_{2n},
    \]
    which is a contradiction by independence.
    Thus $X$ is countable.
    
    Now $G = \bigcup_F \ev{H_x}_{x \in F}$,
    where the union is taken over all finite $F \subseteq X$,
    so since $X$ is countable,
    there is some $F$ for which
    $H_F := \ev{H_x}_{x \in F}$ is non-meager,
    and thus open,
    since $H_F$ is analytic.
    Then $G/H_F$ is countable,
    so if $x\notin F$,
    then $H_x$ is countable by independence.
\end{proof}
In particular,
this implies an unpublished result of Ben Miller
showing that an uncountable-dimensional Polish vector space
does not have an analytic basis.

If $M \sqsubseteq^R N$ and $N \sqsubseteq^R M$,
then we say that $M$ and $N$ are \textbf{bi-embeddable}.
Note that if $M$ and $N$ are $R$-modules,
and $S$ is a subring of $R$,
then $M \sqsubseteq^R N$ implies $M \sqsubseteq^S N$.
In particular,
if $M$ and $N$ are $\sqsubseteq^S$-incomparable,
then they are $\sqsubseteq^R$-incomparable.
In general,
the preorder $\sqsubseteq^R$ can contain incomparable elements.
For example,
$\R$ is $\sqsubseteq^\Z$-incomparable
with the $p$-adic rationals $\Q_p$,
for any prime $p$.
To see this,
we have $\R \not \sqsubseteq^\Z \Q_p$
since $\R$ is connected,
but $\Q_p$ is totally disconnected.
On the other hand,
$\Q_p \not \sqsubseteq^\Z \R$
since $\Q_p$ has a nontrivial compact subgroup,
but $\R$ does not.
So $\R$ and $\Q_p$ are $\sqsubseteq^\Z$-incomparable,
and thus also $\sqsubseteq^\Q$-incomparable.

For certain rings,
no locally compact module embeds into $R^\N$,
and thus a minimum for $\sqsubseteq^R$ cannot be locally compact:
\begin{prop}\label{locallyCompact}
    Let $R$ be a Polish ring
    with no nontrivial compact subgroups,
    and let $M$ be a locally compact Polish $R$-module.
    If $M \sqsubseteq^R R^\N$,
    then $M$ is countably generated.
\end{prop}
\begin{proof}
    Fix a continuous linear injection $f \colon M \hra R^\N$.
    Since $R$ has no nontrivial compact subgroups,
    the same holds for $R^\N$,
    and thus for $M$.
    Fix a complete norm $\|\cdot\|$ compatible with $(M, +)$.
    Let $\pi_n \colon R^\N \to R^n$ denote
    the projection to the first $n$ coordinates,
    and let $M_n = \ker(\pi_n \circ f)$,
    which is a closed submodule of $M$.
    Fix $\ve$ such that the closed $\ve$-ball around $0 \in M$ is compact,
    and let $C = \{m \in M : \frac{\ve}{2} \le \|m\| \le \ve\}$.
    Then $C \cap \bigcap_n M_n = \varnothing$,
    so since $C$ is compact,
    there is some $n$ such that $C \cap M_n = \varnothing$.
    We claim that $M_n$ is discrete.
    To see this,
    suppose that the $\frac{\ve}{2}$-ball around $0 \in M$
    contained some nonzero $m \in M_n$.
    Then the subgroup generated by $m$ is not compact,
    so there is a minimal $k \in \N$ with $\|km\| \ge \frac{\ve}{2}$,
    and hence $km \in C$,
    which is not possible.
    Thus $M_n$ is countable,
    so if we pick preimages $(m_i)_{i < n}$ in $M$
    of the standard basis of $R^n$,
    then $M$ is generated by $M_n \cup (m_i)_{i < n}$,
    and thus countably generated.
\end{proof}
We do not know anything about the preorder $\sqsubseteq^R$
restricted to locally compact modules,
including the existence of a minimum or maximum element.

If $M_0$ and $M_1$ are Polish $R$-modules
with Baire-measurable submodules $N_0$ and $N_1$ respectively,
we write $M_0/N_0 \sqsubseteq^R M_1/N_1$
if there is a continuous linear map $M_0 \to M_1$
which descends to an injection $M_0/N_0 \hra M_1/N_1$.
This map is a Borel reduction
of $E^{M_0}_{N_0}$ to $E^{M_1}_{N_1}$,
where $E^{M_i}_{N_i}$
is the coset equivalence relation of $N_i$ in $M_i$
(see \cite{Gao09} for background on Borel reductions).
In particular,
we have $\R/\Q \not\sqsubseteq^\Q \R$,
since $E^\R_\Q$ is not smooth.
We also have $\R \not\sqsubseteq^\Q \R/\Q$,
since any nontrivial continuous linear map $\R \to \R$ is surjective,
and thus $\R$ and $\R/\Q$ are $\sqsubseteq^\Q$-incomparable.

\section{Proper normed rings}
\label{ProperNorm}
Every proper normed ring is locally compact and Polish.
There are many examples of proper normed rings:
\begin{itemize}
    \item The usual norms on $\Z$, $\R$, $\C$ and $\HH$ are proper.
    \item The $p$-adic norm on $\Q_p$ is proper.
    \item Every countable ring $R$
        admits a proper norm as follows.
        Let $w \colon R \to \N$ be a finite-to-one function
        such that $w(0) = 0$,
        $w(r) \ge 2$ if $r \neq 0$,
        and $w(r) = w(-r)$.
        We extend $w$ to every term $t$
        in the language $({+}, {\cdot}) \cup R$
        by $w(r + s) = w(r) + w(s)$ and $w(r \cdot s) = w(r)w(s)$.
        Then let $|r|$ be the minimum of $w(t)$
        over all terms $t$ representing $r$.
    \item Let $R$ be a proper normed ring.
        If $S \le R$ is a closed subring,
        then there is a proper norm on $S$
        obtained by restricting the norm on $R$.
        If $I \tl R$ is a closed two-sided ideal,
        then there is a proper norm on $R/I$
        given by $|r + I| = \min_{s \in r + I} |s|$.
\end{itemize}
In general,
we do not know if every locally compact Polish ring
admits a compatible proper norm.

Given a closed two-sided ideal $I \tl R$,
there is a natural quotient map
$\ell^1(R) \thra \ell^1(R/I)$
with kernel $\ell^1(I) := \ell^1(R)\cap I^\N$.

If $R$ is finite proper normed ring,
then $\ell^1(R) = R^\N$,
which in particular is homeomorphic to Cantor space.
For infinite discrete rings,
there is also a unique homeomorphism type.
Recall that \textbf{complete Erd\H{o}s space} is
the space of square-summable sequences of irrational numbers
with the $\ell^2$-norm topology.
\begin{prop}\label{l1Erdos}
    Let $R$ be an infinite discrete proper normed ring.
    Then $\ell^1(R)$ is homeomorphic to complete Erd\H{o}s space.
\end{prop}

To show this,
we will use a characterization due to Dijkstra and van Mill
\cite[Theorem 1.1]{DvM09}.
A topological space is \textbf{zero-dimensional}
if it is nonempty and it has a basis of clopen sets.
\begin{thm}[Dijkstra-van Mill]\label{DvM}
    Let $X$ be a separable metrizable space.
    Then $X$ is homeomorphic to complete Erd\H{o}s space iff
    there is a zero-dimensional metrizable topology $\tau$ on $X$
    coarser than the original topology
    such that every point in $X$ has a neighbourhood basis
    (for the original topology)
    consisting of closed nowhere dense Polish subspaces of $(X, \tau)$.
\end{thm}
\begin{proof}[Proof of \Cref{l1Erdos}]
    We check the condition from \Cref{DvM}.
    Let $\tau$ be the product topology on $R^\N$,
    which is zero-dimensional and metrizable.
    It is enough to show that every closed ball is
    a closed nowhere dense Polish subspace of $(\ell^1(R), \tau)$.
    By translation,
    it suffices to consider balls of the form
    $B = \{m \in \ell^1(R) : \|m\| \le \ve\}$.
    Note that $B$ is closed in $R^\N$.
    Thus $(B, \tau)$ is Polish,
    and $B$ is closed in $(\ell^1(R), \tau)$.
    It remains to show that the complement of $B$ is dense in $(\ell^1(R), \tau)$.
    Let $U$ be a nonempty open subset of $(\ell^1(R), \tau)$.
    We can assume that there is a finite sequence $(r_k)_{k < n}$ in $R$
    such that $U$ is the set of sequences in $\ell^1(R)$
    starting with $(r_k)_{k < n}$.
    Since $R$ is infinite and the norm is proper,
    there is some $r \in R$ with $|r| > n!\ve$.
    Then $(r_0, \ldots, r_{n-1}, r, 0, 0, 0, \ldots) \in U\setminus B$.
\end{proof}

\section{Special cases}
For a general Polish ring $R$,
we do not know much about the preorder $\sqsubseteq^R$,
including the following:
\begin{prob}
    Is there a maximum Polish $R$-module under $\sqsubseteq^R$?
\end{prob}
This is known for some particular rings,
which we mention below.

\subsection{Principal ideal domains}
Recall that a \textbf{principal ideal domain (PID)}
is an integral domain in which every ideal is generated by a single element.
There is an irreducible basis for uncountable Polish modules over a PID:
\begin{thm}\label{pid}
    Let $R$ be a proper normed discrete PID
    and let $M$ be a Polish $R$-module.
    Then exactly one of the following holds:
    \begin{enumerate}
        \item $M$ is countable.
        \item There a prime ideal $\p \tl R$ 
            such that $\ell^1(R/\p) \sqsubseteq^R M$.
    \end{enumerate}
    Moreover,
    the $\ell^1(R/\p)$ are $\sqsubseteq^R$-incomparable for different $\p$.
\end{thm}
\begin{proof}
    Suppose that $M$ is not countable.
    By \Cref{discrete},
    there is some proper ideal $I \tl R$ such that
    $\ell^1(R/I) \sqsubseteq^R M$.
    Then since $R$ is a PID,
    there is some prime ideal $\p \tl R$ and some nonzero $s \in R$
    such that $I = \p s$.
    Then the linear injection $R/\p \hra R/I$ defined by $r \mapsto rs$
    induces a continuous linear injection $\ell^1(R/\p) \hra \ell^1(R/I)$.
    
    It remains to show that if $\p$ and $\q$ are prime ideals
    with $\ell^1(R/\p) \sqsubseteq^R \ell^1(R/\q)$,
    then $\p = \q$.
    Fix a continuous linear injection $\ell^1(R/\p) \hra \ell^1(R/\q)$.
    Since $R/\p$ is an integral domain,
    the annihilator of any nonzero element of $\ell^1(R/\p)$ is $\p$,
    and similarly for $\q$.
    Then for any nonzero $x \in \ell^1(R/\p)$,
    its image in $\ell^1(R/\q)$ must have the same annihilator
    since the map is injective,
    and thus $\p = \q$.
\end{proof}

\subsection{Abelian groups}
Applying \Cref{pid} with $R = \Z$ gives
an irreducible basis for uncountable abelian groups:
\begin{thm}\label{abelian}
    Let $A$ be an uncountable abelian Polish group.
    Then one of the following holds:
    \begin{enumerate}
        \item $\ell^1(\Z) \sqsubseteq^\Z A$.
        \item $(\Z/p\Z)^\N \sqsubseteq^\Z A$ for some prime $p$.
    \end{enumerate}
\end{thm}
By \cite{Shk99},
there is a $\sqsubseteq^\Z$-maximum abelian Polish group $A_{\max}$.
So the preorder $\sqsubseteq^\Z$ on
uncountable abelian Polish groups
looks like the following:

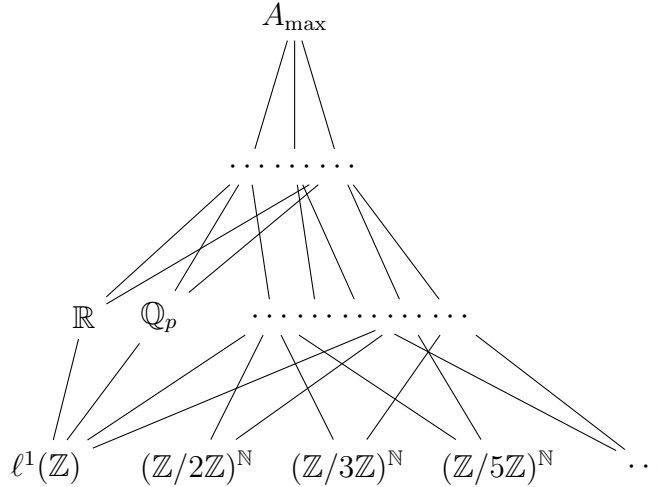
\[
    \begin{tikzpicture}
        \node (l1) at (0,0) {$\ell^1(\Z)$};
        \node (R) at (0.5,2) {$\R$};
        \node (Qp) at (1.5,2) {$\Q_p$};
        \node (mod2) at (2,0) {$(\Z/2\Z)^\N$};
        \node (mod3) at (4,0) {$(\Z/3\Z)^\N$};
        \node (mod5) at (6,0) {$(\Z/5\Z)^\N$};
        \node (etc) at (8,0) {$\cdots$};
        \node (dots0) at (3,2) {$\cdots$};
        \node (dots1) at (3.6,2) {$\cdots$};
        \node (dots2) at (4.2,2) {$\cdots$};
        \node (dots3) at (4.8,2) {$\cdots$};
        \node (dots4) at (5.4,2) {$\cdots$};
        \node (dots5) at (2.7,4) {$\cdots$};
        \node (dots6) at (3.3,4) {$\cdots$};
        \node (dots7) at (3.9,4) {$\cdots$};
        \node (Amax) at (3.3,6) {$A_{\max}$};
        
        \draw (l1) to (R);
        \draw (R) to (dots5);
        \draw (R) to (dots7);
        \draw (l1) to (Qp);
        \draw (Qp) to (dots5);
        \draw (Qp) to (dots7);
        \draw (l1) to (dots0);
        \draw (l1) to (dots3);
        \draw (mod2) to (dots0);
        \draw (mod2) to (dots3);
        \draw (mod3) to (dots0);
        \draw (mod3) to (dots4);
        \draw (mod5) to (dots0);
        \draw (mod5) to (dots3);
        \draw (etc) to (dots2);
        \draw (etc) to (dots4);
        \draw (dots0) to (dots5);
        \draw (dots1) to (dots6);
        \draw (dots2) to (dots6);
        \draw (dots3) to (dots7);
        \draw (dots4) to (dots7);
        \draw (dots5) to (Amax);
        \draw (dots6) to (Amax);
        \draw (dots7) to (Amax);
    \end{tikzpicture}
\]

\subsection{$\Q$-vector spaces}
Fix a proper norm on $\Q$.
By \Cref{locallyCompact},
a $\sqsubseteq^\Q$-minimum uncountable Polish $\Q$-vector space
cannot be locally compact.
By \Cref{discrete},
we have $\ell^1(\Q) \sqsubset^\Q \R$,
where the strictness is due to
$\ell^1(\Q)$ being totally disconnected.
However,
it is open as to whether there is an intermediate vector space:
\begin{prob}
    Is there a Polish $\Q$-vector space $V$ such that
    $\ell^1(\Q) \sqsubset^\Q V \sqsubset^\Q \R$?
\end{prob}

\subsection{Real vector spaces}
We consider the order $\sqsubseteq^\R$
on uncountable-dimensional Polish $\R$-vector spaces.
By \Cref{division},
there is a minimum element $\ell^1(\R)$,
which is bi-embeddable with the usual space $\ell^1$
of absolutely summable sequences.
By \Cref{locallyCompact},
any uncountable-dimensional locally compact Polish $\R$-vector space
must be strictly above $\ell^1$.
By \cite{Kal77},
there is a maximum Polish $\R$-vector space $V_{\max}$.

\section{Proof of the main theorems}
\label{proofs}
Every abelian Polish group $A$
has a compatible complete norm
defined by $\norm{a} = d(a, 0)$,
where $d$ is an invariant metric on $A$
(see \cite[1.1.1, 1.2.2]{BK96}).
If $B \subseteq A$ is a Baire-measurable subgroup,
then by Pettis's lemma,
$B$ is either open or meager
(see \cite[9.11]{Kec95}).

Setting $N = 0$ in the following theorem
recovers \Cref{division}.
\begin{thm}\label{divisionQuotient}
    Let $R$ be a proper normed division ring,
    let $M$ be a Polish $R$-vector space,
    and let $N \subseteq M$ be an $F_\sigma$ vector subspace.
    Then exactly one of the following holds:
    \begin{enumerate}[label=(\arabic*)]
        \item $\dim_R(M/N)$ is countable.
        \item $\ell^1(R) \sqsubseteq^R M/N$.
    \end{enumerate}
\end{thm}
In most natural examples,
$N$ is $F_\sigma$,
such as for $\ell^1(\N) \subseteq \ell^2(\N)$.
It would be interesting to prove this for more general subspaces.

\begin{proof}
    Suppose that the dimension of $M/N$ is uncountable.
    Then $N$ is not open,
    so $N$ is meager,
    i.e.
    we have $N = \bigcup_k F_k$
    for some increasing sequence $(F_k)_k$
    of closed nowhere dense sets.
    Fix a complete norm $\norm{\cdot}$ compatible with $(M, +)$.
    For every $k$,
    we define $\ve_k > 0$ and $m_k \in M$
    such that the image of $(m_k)_k$ in $M/N$
    is linearly independent over $R$.
    We proceed by induction on $k$.
    Choose $\ve_k > 0$ such that
    \begin{enumerate}[label=(\roman*)]
        \item $\ve_k < \frac{1}{2}\ve_i$
            for every $i < k$,
        \item for every $(r_i)_{i < k}$
            such that $\sum_{i < k}\frac{|r_i|}{i!} \le k$
            and there is some $l < k$
            with $r_l = 1$ and $r_i = 0$ for $i < l$,
            the open $\ve_k$-ball
            centered at $\sum_{i < k} r_i m_i$
            is disjoint from $F_k$.
    \end{enumerate}
    Then choose $m_k \in M$ such that
    \begin{enumerate}[label=(\roman*)]
        \item $m_k \notin N + Rm_0 + Rm_1 + \cdots + Rm_{k-1}$,
        \item $\norm{r m_k} < \frac{1}{2}\ve_k$
            whenever $\frac{|r|}{k!} \le k$.
    \end{enumerate}
    We verify that this is possible.
    When choosing $\ve_k$,
    to satisfy the second condition,
    note that the set of considered $(r_i)_{i < k}$ is compact,
    so the set of $\sum_{i < k} r_i m_i$ is also compact,
    and it is disjoint from $N$
    (and hence $F_k$)
    by the choice of $(m_i)_{i < k}$.
    Thus such an $\ve_k$ must exist.
    When choosing $m_k$,
    note that the first condition holds for a comeager set of $m_k$,
    since $N + Rm_0 + Rm_1 \cdots + Rm_{k-1}$ is analytic,
    and it is not open,
    since otherwise $M/N$ would have countable dimension.
    The second condition holds for an open set of $m_k$,
    since the set of $r$ with $\frac{|r|}{k!} \le k$ is compact.
    Thus such an $m_k$ must exist.
    
    We define a map $\ell^1(R) \hra M$ by
    \[
        (r_k)_k \mapsto \sum_k r_k m_k.
    \]
    First we show that this is well-defined,
    from which linearity and continuity are immediate.
    Let $(r_k)_k \in \ell^1(R)$ be nonzero.
    By scaling,
    we can assume that there is some $l$ such that
    $r_l = 1$ and $r_i = 0$ for $i < l$.
    Let $n > l$ be sufficiently large
    such that $\sum_k \frac{|r_k|}{k!} \le n$
    and $0 \in F_n$.
    Then
    \[
        \ve_n \le \norm{\sum_{k < n} r_k m_k}.
    \]
    For every $i$,
    we have $\norm{r_{n+i} m_{n+i}} < \frac{1}{2} \ve_{n+i}$,
    and thus $\norm{r_{n+i} m_{n+i}} < \frac{1}{2^{i+1}} \ve_n$
    by inductively using
    $\ve_{k+1} < \frac{1}{2} \ve_k$.
    Thus
    \[
        \norm{r_{n+i} m_{n+i}}
        < \frac{1}{2^{i+1}}\norm{\sum_{k < n} r_k m_k}.
    \]
    Thus $\sum_k r_k m_k$ is well-defined with
    \[
        \norm{\sum_k r_k m_k}
        < 2\norm{\sum_{k < n} r_k m_k}.
    \]
    
    It remains to show that the induced map
    $\ell^1(R)\to M/N$ is an injection.
    Let $(r_k)_k \in \ell^1(R)$ be nonzero.
    By scaling,
    we can assume that there is some $l$ such that
    $r_l = 1$ and $r_i = 0$ for $i < l$.
    Suppose that $n > l$ is sufficiently large such that
    $\sum_k \frac{|r_k|}{k!} \le n$.
    Since $\norm{r_{n+i} m_{n+i}} < \frac{1}{2^{i+1}} \ve_n$,
    we have $\sum_{i \ge 0} \norm{r_{n+i} m_{n+i}} < \ve_n$,
    and so $\sum_k r_k m_k \notin F_n$.
    This holds for all sufficiently large $n$,
    so $\sum_k r_k m_k \notin N$.
\end{proof}

We recover \cite[Theorem 24]{Mil12} for proper normed division rings:
\begin{cor}[Miller]\label{miller}
    Let $R$ be a proper normed division ring,
    and let $M$ be a Polish $R$-module.
    If $\dim_R(M)$ is uncountable,
    then there is a linearly independent perfect subset of $M$.
\end{cor}
\begin{proof}
    By \Cref{division},
    we can assume that $M = \ell^1(R)$.
    Fix an enumeration $(q_n)_{n \in \N}$ of $\Q$.
    For every $x \in \R$,
    define $\chi_x \in \ell^1(R)$ by
    \[
        (\chi_x)_n =
        \begin{cases}
            1 & q_n < x \\
            0 & \text{otherwise}
        \end{cases}
    \]
    Then $(\chi_x)_{x \in \R}$ 
    is an uncountable linearly independent Borel subset of $\ell^1(R)$,
    so we are done by taking any perfect subset of this.
\end{proof}

There is an analogous generalization of \Cref{discrete}.

\begin{thm}\label{discreteQuotient}
    Let $R$ be a left-Noetherian discrete proper normed ring,
    let $M$ be a Polish $R$-module,
    and let $N \subseteq M$ be an $F_\sigma$ submodule.
    Then exactly one of the following holds:
    \begin{enumerate}[label=(\arabic*)]
        \item $M/N$ is countable.
        \item $\ell^1(R)/\ell^1(I) \sqsubseteq^R M/N$
            for some {proper\footnotemark} two-sided ideal $I \tl R$.
            In particular,
            there is a linear injection $\ell^1(R/I) \hra M/N$.
    \end{enumerate}
\end{thm}
\footnotetext{By proper,
we mean a proper subset
(no relation to proper norms).}
 
\begin{proof}
    Suppose that $M/N$ is not countable.
    Then $N$ is not open,
    and thus meager.
    Let $(U_k)_k$ be a descending neighborhood basis of $0 \in M$,
    and let $I_k = \{r \in R : rU_k \subseteq N\}$.
    Then $(I_k)_k$ is an increasing sequence of ideals,
    so since $R$ is left-Noetherian,
    this sequence stabilizes at some $I = I_n$.
    Note that $I$ is a proper ideal,
    since otherwise $U_n \subseteq N$,
    a contradiction to $N$ being meager.
    Note also that $I$ is a two-sided ideal,
    since if $r \in R$,
    then there is some $k > n$ with $rU_k \subseteq U_n$,
    and thus $IrU_k \subseteq IU_n \subseteq N$,
    and thus $Ir\subseteq I$.
    By replacing $M$ with the submodule generated by $U_n$
    (which is analytic non-meager, and therefore open),
    we can assume that for every nonempty open $V \subseteq M$,
    we have $\{r \in R : rV \subseteq N\} = I$.
    Then for every $r \notin I$,
    the subgroup $\{m \in M : rm \in N\}$ is not open,
    and therefore meager.
    Thus more generally,
    if $m' \in M$,
    then $\{m \in M: rm \in N + m'\}$ is meager.
    
    Fix a complete norm $\norm{\cdot}$ compatible with $(M, +)$.
    Let $(F_k)_k$ be an increasing sequence
    of closed nowhere dense sets
    with $N = \bigcup_k F_k$.
    For every $k$,
    we define $\ve_k > 0$ and $m_k \in M$
    such that the image of $(m_k)_k$ in $M/N$
    is linearly independent over $R/I$.
    We proceed by induction on $k$.
    Choose $\ve_k > 0$ such that
    \begin{enumerate}[label=(\roman*)]
        \item $\ve_k < \frac{1}{2}\ve_i$
            for every $i < k$,
        \item for every $(r_i)_{i < k}$ with
            $\sum_{i < k} r_i m_i$ nonzero
            and $\sum_{i < k}\frac{|r_i|}{i!} \le k$,
            we have $\ve_k \le \norm{\sum_{i < k} r_i m_i}$,
        \item for every $(r_i)_{i < k}$ with
            $\sum_{i < k} r_i m_i \notin N$
            and $\sum_{i < k}\frac{|r_i|}{i!} \le k$,
            the open $\ve_k$-ball
            centered at $\sum_{i < k} r_i m_i$
            is disjoint from $F_k$.
    \end{enumerate}
    Then choose $m_k \in M$ such that
    \begin{enumerate}[label=(\roman*)]
        \item $r m_k \notin N + Rm_0 + Rm_1 + \cdots + Rm_{k-1}$
            for every $r \notin I$,
        \item $\norm{r m_k} < \frac{1}{2}\ve_k$
            whenever $\frac{|r|}{k!} \le k$.
    \end{enumerate}
    We verify that this is possible.
    When choosing $\ve_k$,
    for the second and third condition,
    there is only a finite set of $\sum_{i < k} r_i m_i$ to consider,
    and for the third condition,
    this set is disjoint from $N$,
    and hence from $F_k$.
    Thus such an $\ve_k$ must exist.
    When choosing $m_k$,
    for the first condition,
    for a fixed $r \notin I$ and $m' \in Rm_0 + \cdots + Rm_{k-1}$,
    we have shown earlier that $\{rm \notin N + m'\}$ is comeager,
    so by quantifying over the countably many $r$ and $m'$,
    the set of $m_k$ satisfying the first condition is comeager.
    The second condition holds for an open set of $m_k$,
    since the set of $r$ with $\frac{|r|}{k!} \le k$ is finite.
    Thus such an $m_k$ must exist.
    
    We define a map $\ell^1(R) \hra M$ by
    \[
        (r_k)_k \mapsto \sum_k r_k m_k.
    \]
    First we show that this is well-defined,
    from which linearity and continuity are immediate.
    Let $(r_k)_k \in \ell^1(R)$.
    We can assume that there is some $n$ such that
    $\sum_{k < n} r_k m_k$ is nonzero
    and $\sum_{k < n} \frac{|r_k|}{k!} \le n$.
    Then
    \[
        \ve_n \le \norm{\sum_{k < n} r_k m_k}.
    \]
    For every $i$,
    we have $\norm{r_{n+i} m_{n+i}} < \frac{1}{2} \ve_{n+i}$,
    and thus $\norm{r_{n+i} m_{n+i}} < \frac{1}{2^{i+1}} \ve_n$
    by inductively using
    $\ve_{k+1} < \frac{1}{2} \ve_k$.
    Thus
    \[
        \norm{r_{n+i} m_{n+i}}
        < \frac{1}{2^{i+1}}\norm{\sum_{k < n} r_k m_k}.
    \]
    Thus $\sum_k r_k m_k$ is well-defined with
    \[
        \norm{\sum_k r_k m_k}
        < 2\norm{\sum_{k < n} r_k m_k}.
    \]
    
    It remains to show that the kernel of
    the induced map $\ell^1(R)\to M/N$ is $\ell^1(I)$.
    The kernel clearly contains $\ell^1(I)$,
    since $IM \subseteq N$.
    Now let $(r_k)_k \in \ell^1(R) \setminus \ell^1(I)$.
    Since the image of $(r_k)_k$ in $M/N$
    is linearly independent over $R/I$,
    if $n$ is sufficiently large,
    then $\sum_{k < n} r_k m_k \notin N$
    and $\sum_k \frac{|r_k|}{k!} \le n$.
    Since $\norm{r_{n+i} m_{n+i}} < \frac{1}{2^{i+1}} \ve_n$,
    we have $\sum_{i \ge 0} \norm{r_{n+i} m_{n+i}} < \ve_n$,
    and so $\sum_k r_k m_k \notin F_n$.
    This holds for all sufficiently large $n$,
    so $\sum_k r_k m_k \notin N$.
\end{proof}

\bibliographystyle{alpha}
\bibliography{polish_modules}
\end{document}